\documentclass[leqno,12pt]{amsart} % leqno\UTF{0082}\UTF{00C5}\UTF{0090}\UTF{0094}\UTF{008E}\UTF{00AE}\UTF{0094}\UTF{00D4}\UTF{008D}\UTF{0086}\UTF{0082}\UTF{00F0}\UTF{008D}¶\UTF{0091}\UTF{20AC}\UTF{0082}\UTF{00D6}\UTF{0094}z\UTF{0092}u
\usepackage{amssymb}
\usepackage{enumerate}  %% 
\usepackage{graphicx}  %% 
\usepackage{epstopdf}
\usepackage{comment}
\usepackage{array}

 \makeatletter
    
    \@addtoreset{equation}{section}
  \makeatother
\theoremstyle{plain} % \UTF{0083}C\UTF{0083}^\UTF{0083}\UTF{008A}\UTF{0083}b\UTF{0083}N\UTF{0091}\UTF{00CC}%
\newtheorem{theorem}{\indent\sc Theorem}[section] % \UTF{008C}\UTF{00A9}\UTF{008F}o\UTF{0082}\UTF{00B5}\UTF{0082}\UTF{00CD}\UTF{0083}X\UTF{0083}\UTF{0082}\UTF{0081}[\UTF{0083}\UTF{008B}\UTF{0083}L\UTF{0083}\UTF{0083}\UTF{0083}b\UTF{0083}v

\newtheorem{corollary}[theorem]{\indent\sc Corollary}
\newtheorem{proposition}[theorem]{\indent\sc Proposition}

\theoremstyle{definition} % \UTF{0083}\UTF{008D}\UTF{0081}[\UTF{0083}}\UTF{0083}\UTF{0093}\UTF{0091}\UTF{00CC}\UTF{0082}\UTF{00C9}\UTF{0095}\UTF{00CF}\UTF{008D}X
\newtheorem{definition}[theorem]{\indent\sc Definition}
\newtheorem{remark}[theorem]{\indent\sc Remark}

\DeclareMathOperator{\diam}{diam}
\begin{document}

\title{\uppercase{A construction of graphs with positive Ricci curvature}} %論文タイトル大文字
\author{\textsc{Taiki Yamada}}
\date{} %日付けは記入しない
%

%%%%%%%%%%%%%%%%%%% \UTF{008B}r\UTF{0092}\UTF{008D} %%%%%%%%%%%%%%%%%%%%%%%%%%%%%%

\footnote{ %2010MSC
}


\footnote{
This work was supported by Grant-in-Aid for JSPS Research Fellow Grant Number 18J10494.
}

%%%%%%%%%%%% \UTF{0092}\UTF{0098}\UTF{008E}\UTF{00D2}\UTF{008F}\UTF{008A}\UTF{0091}\UTF{00AE} %%%%%%%%%%%%%
\address{ 
Mathematical Institute in Tohoku University \endgraf
Sendai 980-8578 \endgraf
Japan
}
\email{mathyamada@dc.tohoku.ac.jp}

%%%%%%%%%%%%%%%%%%%%%%%%%%%%%%%%%%%%%%%%%%%%%%%%%%%%%%%

\maketitle

\begin{abstract}
Two complete graphs are connected by adding some edges. The obtained graph is called the gluing graph. The more we add edges, the larger the Ricci curvature on it becomes. We calculate the Ricci curvature of each edge on the gluing graph and obtain the least number of edges that result in the gluing graph having positive Ricci curvature.
\end{abstract}

\section{Introduction}
 The Ricci curvature is one of the most important concepts in Riemannian geometry. There are some definitions for the generalized Ricci curvature, and Ollivier's coarse Ricci curvature is one of them (see \cite{Ol1}). It is formulated using the Wasserstein distance on a metric space $(X, d)$ with a random walk $m=\left\{m_{x} \right\}_{x \in X}$, where each $m_{x}$ is a probability measure on $X$. The coarse Ricci curvature is defined as follows: For two distinct points $x, y \in X$, 
	\begin{eqnarray*}
	\kappa(x, y) := 1 - \cfrac{W(m_{x}, m_{y})}{d(x, y)},
	\end{eqnarray*}
  where $W$($m_{x}, m_{y}$) is the $1$-Wasserstein distance between $m_{x}$ and $m_{y}$. This definition was applied to graphs in the year 2010 and many researchers are focused on this. \\
　In 2010, Lin-Lu-Yau \cite{Yau1} defined the Ricci curvature of an undirected graph by using the coarse Ricci curvature of the lazy random walk. They studied the Ricci curvature of the product space of graphs and random graphs. They also considered a graph of positive Ricci curvature and proved some properties. In 2012, Jost and Liu \cite{Jo2} studied the relationship between the coarse Ricci curvature of the simple random walk and the local clustering efficient. In this paper, the coarse Ricci curvature is simply called the {\em Ricci curvature}. Given that we obtain several global properties on graphs with positive Ricci curvature by their results, we would like to constitute a new graph with positive Ricci curvature. A complete graph $K_n$ is one of graphs with positive Ricci curvature. In fact, we have $\kappa(x, y) = (n-2)/(n-1)$ for any edge $(x, y) \in E(K_n)$. Therefore, in this paper, we start from two complete graphs $K_n$ and $K'_n$, and connect them by several edges. For the obtained graph that is called the $m$-gluing graph (See Definition \ref{gluing}), we obtain the following theorem regarding the minimal number of $m$ satisfying the condition that the $m$-gluing graph must have positive Ricci curvature.
\begin{theorem}
\label{主結果}
For the $m$-gluing graph $K_{n} +_m K'_{n}$ of two complete graphs $K_{n}$ and $K'_{n}$, $n \geq 5$, we have $\kappa(x, y) > 0$ for any edge $(x, y) \in E$ if and only if 
\begin{eqnarray*}
\cfrac{n^{2} - 2n}{n + 2} < m \leq n - 1.
\end{eqnarray*}
%Especially if $n \geq 5$, we have $\kappa(K_{n} +_m K'_{n}) > 0$.
\end{theorem}
If we calculate the Ricci curvature of edges on the $m$-gluing graph by the estimates of known results (see Theorem \ref{Jost-Liuの評価} and Theorem \ref{Jost-Liuの評価２}), then there exist several edges such that the estimate of either Theorem \ref{Jost-Liuの評価} or Theorem \ref{Jost-Liuの評価２} is not optimal (see Remark \ref{example1} and Remark \ref{example2}). Thus we calculate the Ricci curvature of each edge based on the definition of the Ricci curvature.\\
　In the last section, we obtain some estimates of the eigenvalues of the normalized graph Laplacian and the Cheeger constant by the Ricci curvature on the $m$-gluing graph.

\section*{acknowledgment} 
The author thanks his supervisor, Professor Takashi Shioya, for his continuous support and providing important comments. He also thanks Editage (www.editage.jp) for English language editing.

\section{Ricci curvature}
In this paper, we assume that $G=(V, E)$ is an undirected connected simple finite graph, where $V$ is the set of the vertices and $E$ the set of edges. That is,
	\begin{enumerate}
	\item for any two vertices, there exists a path connecting them,
	\item there exists no loop and no multiple edges,
	\item and the number of vertices and edges is finite.
	\end{enumerate}
For $x, y \in V$, we write $(x, y)$ as an edge from $x$ to $y$. We denote the set of vertices of a graph $H$ by $V(H)$ and the set of edges of $H$ by $E(H)$.

%距離の定義%
\begin{definition}
	\begin{enumerate}
  	\renewcommand{\labelenumi}{(\arabic{enumi})}
	\item A {\em path} from a vertex $x \in V$ to a vertex $y \in V$ is a sequence of edges $\left\{(a_{i}, a_{i+1)} \right\}_{i=0}^{n-1}$, where $a_{0} = x$, $a_{n} = y$. We call $n$ the {\em length} of the path.
   	\item The {\em distance} $d(x, y)$ between two vertices $x, y \in V$ is given by the length of a shortest path between $x$ and $y$.
	\item The {\em diameter} of $G$, denoted by $\diam G$, is given by the maximum of the distance between any two vertices on $G$.
   	\end{enumerate}
\end{definition}

%次数の定義%
\begin{definition}
	\begin{enumerate}
 	\renewcommand{\labelenumi}{(\arabic{enumi})}
 	\item For any $x \in V$, the {\em neighborhood} of $x$ is defined as
   		\begin{eqnarray*}
   		\Gamma(x) := \left\{y \in V \mid (x, y) \in E \right\}.
   		\end{eqnarray*}
  	\item For any $x \in V$, the {\em degree} of $x$, denoted by $d_{x}$, is the number of edges starting from $x$.
	\item If every vertex has the same degree $d$, then we call $G$ a $d$-{\em regular graph}.
   	\end{enumerate}
\end{definition}

%確率測度の定義%
\begin{definition}
For any vertex $x \in V$, we define a probability measure $m_{x}$ on $V$ by
	\begin{eqnarray*}
    	m_{x}(v) = 
    		\begin{cases}
    		\cfrac{1}{d_{x}}, &\mathrm{if}\ (x, v) \in E, \\
    		0, & \mathrm{\mathrm{otherwise}}.
    		\end{cases}
   	\end{eqnarray*}
\end{definition} 
$m=\left\{ m_x \right\}_{x \in V}$ is called the {\em simple random walk}.

%輸送距離の定義%
To define the Ricci curvature, we define the 1-Wasserstein distance as follows:
\begin{definition}
The 1-Wasserstein distance between any two probability measures $\mu$ and $\nu$ on $V$ is given as follows
	\begin{eqnarray*}
	W(\mu, \nu) = \inf_{A} \sum_{u, v \in V}A(u, v)d(u, v),
	\end{eqnarray*}
	where $A : V(G) \times V(G) \to [0, 1]$ runs over all maps satisfying 
	\begin{eqnarray}
	\label{coupling}
		\begin{cases}
		\sum_{v \in V}A(u, v) = \mu(u),\\
		\sum_{u \in V}A(u, v) = \nu(v).
		\end{cases}
	\end{eqnarray}
Such a map $A$ is called a {\em coupling} between $\mu$ and $\nu$. 
\end{definition}
	\begin{remark}
	There exists a coupling $A$ that attains the 1-Wasserstein distance (see \cite{Le}, \cite{Vi1} and \cite{Vi2}), and we call it an {\em optimal coupling}.
	\end{remark}
	
   One of the most important properties of the 1-Wasserstein distance is the Kantorovich-Rubinstein duality, which is stated as follows:
	\begin{proposition}[Kantorovich, Rubinstein]
	\label{kantoro}
	The 1-Wasserstein distance between any two probability measures $\mu$ and $\nu$ on $V$ is written as
		\begin{eqnarray*}
		W(\mu, \nu) = \sup_{f} \sum_{u \in V} f(u)(\mu(u) - \nu(u)),
		\end{eqnarray*}
	where the supremum is taken over all functions on $G$ that satisfy $|f(u)-f(v)| \leq d(u,v)$ for any $u, v \in V$. 
	\end{proposition}
A function $f$ on $V$ is said to be {\em $1$-Lipschitz} if $|f(u)-f(v)| \leq d(u,v)$ for any $u, v \in V$.
%Ricci 曲率の定義%
\begin{definition}
	\label{Ricci}
For any two distinct vertices $x, y \in V$, the {\em Ricci curvature} of $x$ and $y$ is defined as follows:
   	\begin{eqnarray*}
    	\kappa(x, y) =  1 - \cfrac{W(m_{x}, m_{y})}{d(x, y)}.
   	\end{eqnarray*}
\end{definition}

Whenever we are interested in a lower bound of Ricci curvature, the following lemma implies that it is sufficient to consider the Ricci curvature of the edge, although the Ricci curvature is defined for any pair of vertices.
\begin{proposition}[Lin-Lu-Yau, \cite{Yau1}]
If $\kappa(u, v) \geq k$ for any edge $(u, v) \in E$ and for a real number $k$, then $\kappa(x, y) \geq k$ for any pair of vertices $(x, y) \in V \times V$.
\end{proposition}

To calculate the Ricci curvature, Jost and Liu gave the following estimate of the Ricci curvature.
\begin{theorem}[Jost-Liu \cite{Jo2}]
\label{Jost-Liuの評価}
On a locally finite graph, for any pair of neighboring vertices $x$ and $y$, we have
\begin{eqnarray*}
\kappa(x, y) \geq - \left( 1- \cfrac{1}{d_x} - \cfrac{1}{d_y} - \cfrac{\# (x, y)}{d_x \wedge d_y} \right)_+ - \left( 1- \cfrac{1}{d_x} - \cfrac{1}{d_y} - \cfrac{\# (x, y)}{d_x \vee d_y} \right)_+ + \cfrac{\# (x, y)}{d_x \vee d_y},
\end{eqnarray*}
where $\# (x, y)$ is the number of triangles which includes $x, y$ as vertices, and $s_+ := \max (s, 0)$, $s \vee t := \max (s, t)$, and $s \wedge t := \min (s, t)$ for real numbers $s$ and $t$.
\end{theorem}
\begin{theorem}[Jost-Liu \cite{Jo2}]
\label{Jost-Liuの評価２}
On a locally finite graph, for any pair of neighboring vertices $x$ and $y$, we have
\begin{eqnarray*}
\kappa(x, y) \leq \cfrac{\#(x,y)}{d_x \vee d_y}.
\end{eqnarray*}
\end{theorem}
 
\section{Proof of the main result}
Let $K_n$ and $K'_n$ be complete graphs, where $V(K_n) = \left\{ u_0, u_1, \cdots, u_{n-1} \right\}$ and $V(K'_n) = \left\{ v_0, v_1, \cdots, v_{n-1} \right\}$. Before we prove Theorem \ref{主結果}, we define the $m$-gluing graph of $K_n$ and $K'_n$ as follows: 
\begin{definition}
\label{gluing}
The {\em $m$-gluing graph} of two complete graphs $K_n$ and $K'_n$, denoted by $K_{n} +_m K'_{n}$, is defined by the following.
\begin{enumerate}
\item The vertex set $V(K_{n} +_m K'_{n}) = V(K_n) \cup V(K'_n)$.
\item The edge set $E(K_n +_m K'_n) = E(K_n) \cup E(K'_n) \cup \left\{ (u_0, v_0) \right\} \cup \left\{ (u_0, v_i) \mid 1 \leq i \leq m \right\} \cup \left\{ (u_j , v_0) \mid 1 \leq j \leq m \right\}$.
\end{enumerate}
We denote $\left\{ v_i \in V(K'_n) \mid 1 \leq i \leq m \right\}$ and $\left\{ u_j \in V(K_n) \mid 1 \leq j \leq m \right\}$ by $\Gamma_m (u_0)$ and $\Gamma_m (v_0)$, respectively.
\end{definition}
\subsection{Proof of Theorem \ref{主結果} in the case of $\bf m=n-1$}
To obtain the value of $\kappa(x, y)$ on any edge $(x, y) \in E(K_{n} +_{n-1} K'_{n})$, it is sufficient to calculate $\kappa(u_0, v_0)$, $\kappa(u_0,v_1)$, $\kappa(u_0,u_1)$, and $\kappa(u_1, u_2)$ by the symmetry of the $m$-gluing graph. Thus, combining Theorem \ref{n-1の場合のRicci曲率} and Proposition \ref{３番目}, we obtain $\kappa(x, y) > 0$ for any edge $(x, y) \in E(K_{n} +_{n-1} K'_{n})$.
\begin{theorem}
\label{n-1の場合のRicci曲率}
Regarding the gluing graph $K_{n} +_{n-1} K'_{n}$, we have
\begin{enumerate}\rm 
\item $\kappa(u_1, u_2) = \cfrac{n-1}{n}$,
\item $\kappa(u_0,v_0) = \cfrac{2n-2}{2n-1}$,
\item $\kappa(u_0, v_1) = \cfrac{3n - 2}{n(2n-1)}$.
\end{enumerate}
\end{theorem}
  \begin{figure}[h]
     \label{glue} %ラベルをつけ図の参照を可能にする
     \begin{center} 
     \includegraphics[scale=0.34]{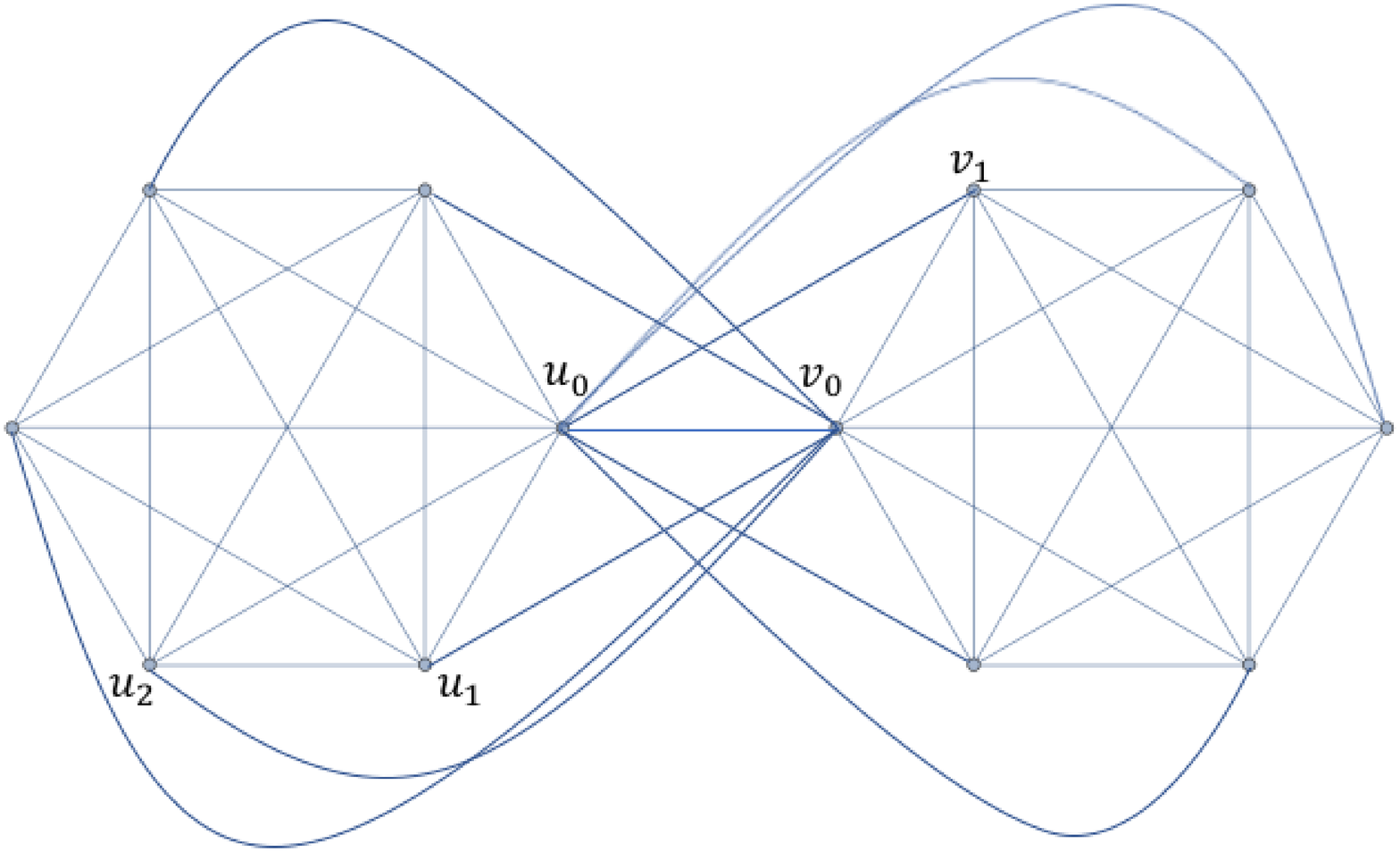}
     \caption{$K_{6} +_{5} K_{6}$}
     \end{center}
    \end{figure}
\begin{proof}
Given that (1) and (2) are held by using Theorem \ref{Jost-Liuの評価} and Theorem \ref{Jost-Liuの評価２}, it is sufficient to prove only (3). To calculate the Ricci curvature on $(u_0, v_1)$, we define a map $B_3 : V \times V \to \mathbb{R}$ by
\begin{eqnarray*}
B_3 (x, y)=
\begin{cases}
\cfrac{1}{2n-1}, & \mathrm{if}\ x=y \in \Gamma(u_0)\cap \Gamma(v_1),\\
\cfrac{1}{(n-2)(2n-1)}, & \mathrm{if}\ x=v_1, y \in \Gamma_{n-1} (u_0) \setminus \left\{ v_1 \right\},\\
\cfrac{1}{n( n-1 )}, & \mathrm{if}\ x \in \Gamma_{n-1} (v_0), y=u_0,\\
\cfrac{1}{n-1} \left( \cfrac{1}{n} - \cfrac{1}{2n-1} \right), & \mathrm{if}\ x \in \Gamma_{n-1} (v_0), y = v_0,\\
\cfrac{1}{n-1} \left\{ \left( \cfrac{1}{n} - \cfrac{1}{2n-1} \right) - \cfrac{1}{(n-2)(2n-1)} \right\}, & \mathrm{if}\ x \in \Gamma_{n-1} (v_0), y \in \Gamma_{n-1} (u_0) \setminus \left\{ v_1 \right\}, \\
0, & \mathrm{otherwise}.
\end{cases}
\end{eqnarray*}
This map is a coupling between $m_{u_0}$ and $m_{v_1}$. In fact, we have
\begin{eqnarray*}
\sum_{x \in V} B_3 (x, v_j) &=& \cfrac{1}{2n-1} + \cfrac{1}{(n-2)(2n-1)}+ \left\{ \left( \cfrac{1}{n} - \cfrac{1}{2n-1} \right) - \cfrac{1}{(n-2)(2n-1)} \right\}\\ 
&=& \cfrac{1}{n},
\end{eqnarray*}
\begin{eqnarray*}
\sum_{y \in V} B_3 (u_i, y) &=& \cfrac{1}{n( n-1 )}+ \cfrac{1}{n-1} \left( \cfrac{1}{n} - \cfrac{1}{2n-1} \right)\\
&+& \cfrac{n-2}{n-1} \left\{ \left( \cfrac{1}{n} - \cfrac{1}{2n-1} \right) - \cfrac{1}{(n-2)(2n-1)} \right\} \\
&=& \cfrac{1}{2n-1}.
\end{eqnarray*}
The others are obvious. Then, the 1-Wasserstein distance is estimated as follows:
\begin{eqnarray*}
W(m_{u_0}, m_{v_1}) &\leq& \cfrac{1}{n} +  \left( \cfrac{1}{n} - \cfrac{1}{2n-1} \right) +  \cfrac{1}{2n-1} \\
&+& (n-2)\left\{ \left( \cfrac{1}{n} - \cfrac{1}{2n-1} \right) - \cfrac{1}{(n-2)(2n-1)} \right\}\\
&=& \cfrac{2n-2}{n} -\cfrac{2n-2}{2n-1}.
\end{eqnarray*}
On the other hand, we define a function $f : V \to \mathbb{R}$ by
\begin{eqnarray*}
f(w)=
\begin{cases}
-2, & \mathrm{if}\ w \in \Gamma_{n-1} (u_0) \setminus \left\{ v_1 \right\},\\
-1, & \mathrm{if}\ w \in \left\{u_0, v_1, v_0 \right\},\\
0, & \mathrm{otherwise}.
\end{cases}
\end{eqnarray*}
It is easy to show that $f$ is a 1-Lipschitz function. By Proposition \ref{kantoro}, the 1-Wasserstein distance is estimated as follows:
\begin{eqnarray*}
W(m_{u_0}, m_{v_1}) &\geq& (n - 2) \left\{ 2 \times \left( \cfrac{1}{n} - \cfrac{1}{2n-1} \right) \right\} + \cfrac{1}{n} - \cfrac{1}{2n-1} + \left( \cfrac{1}{n} - \cfrac{1}{2n-1} \right) \\
&=& (2n-2) \left( \cfrac{1}{n} - \cfrac{1}{2n-1} \right).
\end{eqnarray*}
Thus we obtain 
\begin{eqnarray*}
W(m_{u_0}, m_{v_1}) = (2n-2) \left( \cfrac{1}{n} - \cfrac{1}{2n-1} \right) = \cfrac{2 (n-1)^2}{n(2n-1)},
\end{eqnarray*}
which implies $\kappa(u_0, v_1) = (3n - 2)/\left\{n(2n-1)\right\}$. This completes the proof.
\end{proof}

%３番目のリッチ曲率%
For the edge $(u_0, u_1)$, we calculate the Ricci curvature for $m \leq n-1$.
\begin{proposition}
\label{３番目}
For any vertex $u \in \Gamma_m (v_0)$, we have
\begin{eqnarray*}
\kappa(u_0, u) = 
\begin{cases}
\cfrac{n-1}{n+1}, & \mathrm{if}\ m=1,\\
\cfrac{n^{2} - mn + 2m}{n (n+m)}, & \mathrm{if}\ m \geq 2.
\end{cases}
\end{eqnarray*}
\end{proposition}
\begin{proof}
We take the vertex $u_1 \in \Gamma_m (v_0)$, and consider the Ricci curvature of $(u_0, u_1)$. 
Given that the case of $m=1$ is held by using Theorem \ref{Jost-Liuの評価} and \ref{Jost-Liuの評価２}, we consider the case of $m \geq 2$. We define a map $A : V \times V \to \mathbb{R}$ by
\begin{eqnarray*}
A (x, y)=
\begin{cases}
\cfrac{1}{n+m}, & \mathrm{if}\ x=y \in \Gamma(u_0)\cap \Gamma(u_1),\\
\cfrac{1}{nm}, & \mathrm{if}\ x \in \Gamma_m (u_0), y=u_0,\\
\cfrac{1}{m} \left( \cfrac{1}{n} - \cfrac{1}{n+m} \right), & \mathrm{if}\ x \in \Gamma_m (u_0), y=v_0,\\
\cfrac{1}{(n-2)(n+m)}, & \mathrm{if}\ x=u_1, y \in V(K_n) \setminus \left\{ u_0, u_1 \right\},\\
\cfrac{1}{m} \left\{ \left( \cfrac{1}{n} - \cfrac{1}{n+m} \right) - \cfrac{1}{(n-2)(n+m)} \right\}, & \mathrm{if}\ x \in \Gamma_m (u_0), y \in V(K_n) \setminus \left\{ u_0, u_1 \right\}, \\
0, & \mathrm{otherwise}.
\end{cases}
\end{eqnarray*}
This map is a coupling between $m_{u_0}$ and $m_{u_1}$. Then, the 1-Wasserstein distance is estimated as follows:
\begin{eqnarray*}
W(m_{u_0}, m_{u_1}) &\leq& \cfrac{1}{n} + \left( \cfrac{1}{n} - \cfrac{1}{n+m} \right) + \cfrac{1}{n+m} \\
&+& 2(n-2)  \left\{ \left( \cfrac{1}{n} - \cfrac{1}{n+m} \right) - \cfrac{1}{(n-2)(n+m)} \right\}\\
&=& \cfrac{2n-2}{n} -\cfrac{2n-2}{n+m} = 1 - \cfrac{n^{2} - mn + 2m}{n (n+m)} .
\end{eqnarray*}
On the other hand, we define a function $f_3$ by
\begin{eqnarray*}
f_3(w)=
\begin{cases}
-2, & \mathrm{if}\ w \in V(K_n) \setminus \left\{ u_0, u_1 \right\},\\
-1, & \mathrm{if}\ w \in \left\{ u_0, u_1, v_0 \right\},\\
0, & \mathrm{otherwise}.
\end{cases}
\end{eqnarray*}
It is easy to show that $f_3$ is a 1-Lipschitz function. By using Proposition \ref{kantoro}, the 1-Wasserstein distance is estimated as follows:
\begin{eqnarray*}
W(m_{u_0}, m_{u_1}) &\geq& -2(n - 2) \left(\cfrac{1}{n+m} - \cfrac{1}{n} \right) + \cfrac{1}{n} - \cfrac{1}{n+m} + \left( \cfrac{1}{n} - \cfrac{1}{n+m} \right) \\
&=& (2n-2) \left( \cfrac{1}{n} - \cfrac{1}{n+m} \right).
\end{eqnarray*}
This completes the proof.
\end{proof}
\begin{remark}
\label{example1}
In the case of $m \geq 2$, if we use Theorem \ref{Jost-Liuの評価}, then we obtain
\begin{eqnarray*}
\kappa(u_0, u) \geq \cfrac{n^2 - nm + m}{n(n+m)},
\end{eqnarray*}
for any vertex $u \in \Gamma_m (v_0)$. Therefore, Proposition \ref{３番目} is stronger than Theorem \ref{Jost-Liuの評価}.
\end{remark}

\subsection{Proof of Theorem \ref{主結果} in the case of $\bf m \leq n-2$}
In the case of $m \leq n-2$, we must calculate the Ricci curvature for 7 edges on $K_{n} +_{m} K'_{n}$ using the symmetry of the $m$-gluing graph (see Figure \ref{case}). In Figure \ref{case}, the Ricci curvature for the edges, as in (3), has already be calculated by using Proposition \ref{３番目}. Combining Proposition \ref{３番目}, \ref{１番目}, \ref{２番目}, \ref{４番目}, and Corollary \ref{５から７番目}, we obtain Theorem \ref{主結果}.
  \begin{figure}[h]
     \begin{center} 
     \includegraphics[scale=0.34]{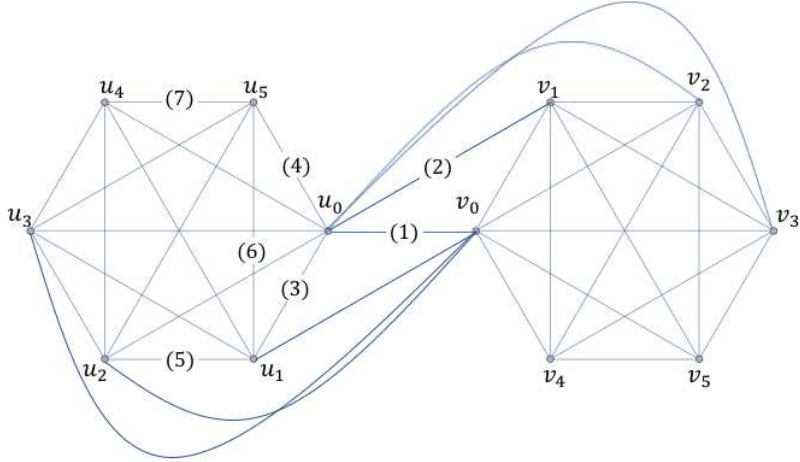}
     \caption{7 patterns in $K_{6} +_{3} K_{6}$}
     \label{case} %ラベルをつけ図の参照を可能にする
     \end{center}
    \end{figure} 

%１番目のリッチ曲率%
For the edge (1) in Figure \ref{case}, we obtain the following proposition.
\begin{proposition}
\label{１番目}
For vertices $u_0$ and $v_0$, we have
\begin{eqnarray*}
\kappa(u_0, v_0) = \cfrac{4m-2n+4}{n+m}.
\end{eqnarray*}
In particular, $\kappa(u_0, v_0) \geq 0$ if and only if $m \geq (n-2)/2$.
\end{proposition}
\begin{proof}
By Theorem \ref{Jost-Liuの評価}, we have
\begin{eqnarray*}
W(m_{u_0}, m_{v_0}) \leq \cfrac{3n-3m-4}{n+m} .
\end{eqnarray*}
On the other hand, we define a map $f_1 : V \to \mathbb{R}$ by
\begin{eqnarray*}
f_1 (w)=
\begin{cases}
3, & \mathrm{if}\ w \in V(K_n) \setminus( \Gamma_m (v_0) \left\{ u_0 \right\} ),\\
2, & \mathrm{if}\ w \in \left\{u_0 \right\} \cup \Gamma_m (v_0),\\
1, & \mathrm{if}\ w \in \Gamma_m (u_0) \cup \left\{ v_0 \right\},\\
0, & \mathrm{otherwise}.
\end{cases}
\end{eqnarray*}
Since $f_1$ is a 1-Lipschitz function, we have
\begin{eqnarray*}
W(m_{u_1}, m_{v_0}) \geq  \cfrac{3}{n+m} (n-m-1) + 2 \left( - \cfrac{1}{n+m} \right) + \cfrac{1}{n+m}= \cfrac{3n-3m-4}{n+m}.
\end{eqnarray*}
Thus we obtain
\begin{eqnarray*}
W(m_{u_1}, m_{v_0}) = \cfrac{3n-3m-4}{n+m}= 1- \cfrac{4m-2n+4}{n+m}.
\end{eqnarray*}
This completes the proof.
\end{proof}
%Jostの結果との比較%
\begin{remark}
The 1-Lipschitz function defined in the proof of Theorem \ref{Jost-Liuの評価} to obtain the upper bound of the Ricci curvature cannot be applied in this case (see \cite{Jo2}); hence, we define the original 1-Lipschitz function $f_1$.
\end{remark}

%２番目のリッチ曲率%
For the edge (2) in Figure \ref{case}, we obtain the following proposition:
\begin{proposition}
\label{２番目}
For any vertex $v \in \Gamma_m (u_0)$, we have
\begin{eqnarray*}
\kappa(u_0, v) = 
\begin{cases}
\cfrac{(2+n)m + 2n - n^2}{n(n+m)}, & \mathrm{if}\ m \geq \cfrac{-n + \sqrt{5n^2 -8n}}{2},\\
\cfrac{m^2 + 2(n+1)m - 2n^2 +4n}{n(n+m)}, & \mathrm{if}\ m < \cfrac{-n + \sqrt{5n^2 -8n}}{2}.
\end{cases}
\end{eqnarray*}
In particular, $\kappa(u_0,v) > 0$ if and only if $m > (n^2 -2n)/(n+2)$.
\end{proposition}
\begin{proof}
We take $v_{1} \in \Gamma_m (u_0)$, and consider the Ricci curvature of $(u_0, v_{1})$. If $m \geq (-n + \sqrt{5n^2 -4n})/2$, then we define a map $A_2: V \times V \to \mathbb{R}$ by
\begin{eqnarray*}
A_2 (x, y)
=
\begin{cases}
\cfrac{1}{n+m}, & \mathrm{if}\ x=y \in \Gamma(u_0) \cap \Gamma(v_{1}),\\
\cfrac{1}{(n-m-1)(n+m)}, & \mathrm{if}\ x=v_{1}, y \in V(K'_n) \setminus \Gamma_m (u_0),\\
\cfrac{1}{n(n-m-1)}, & \mathrm{if}\ x \in V(K_n) \setminus \Gamma_m (v_0), y = u_0,\\
\cfrac{1}{m-1} \left( \cfrac{1}{n+m} -\cfrac{1}{n(n-m-1)} \right), & \mathrm{if}\ x \in V(K_n) \setminus \Gamma_m (v_0),\\ &\ \ \  y \in  \Gamma_m (u_0) \setminus \left\{ v_1 \right\},\\
\cfrac{1}{m} \left( \cfrac{1}{n} - \cfrac{1}{n+m} \right), & \mathrm{if}\ x \in \Gamma_m (v_0), y= v_0,\\
\cfrac{1}{m}\left(\cfrac{1}{n}- \cfrac{1}{(n-m-1)(n+m)} \right), & \mathrm{if}\ x \in \Gamma_m (v_0), y \in V(K'_n) \setminus \Gamma_m (u_0),\\
\cfrac{1}{m} \left\{ \cfrac{1}{n} - \cfrac{1}{n+m} - \cfrac{1}{m-1} \left( \cfrac{n-m-1}{n+m} - \cfrac{1}{n} \right) \right\}, & \mathrm{if}\ x \in \Gamma_m (v_0),\ y \in \Gamma_m (u_0) \setminus \left\{ v_1 \right\},\\
0, & \mathrm{otherwise}.
\end{cases}
\end{eqnarray*}
Since $1/(n+m) \geq 1/n(n-m-1)$ is held for $n-3 \geq m$, we should assume that $n-3 \geq m$, and if $m \geq (-n + \sqrt{5n^2 -4n})/2$, we have $(n-m-1)/(n+m) - 1/n \leq (m-1)(1/n - 1/(n+m))$. So, $A_2$ is a coupling between $m_{u_0}$ and $m_{v_1}$. Then the 1-Wasserstein distance is estimated as follows:
\begin{eqnarray*}
W(m_{u_0}, m_{v_1}) &\leq& \cfrac{1}{n+m} + \cfrac{1}{n} + \left( \cfrac{1}{n} - \cfrac{1}{n+m}\right) + 2(n-m-1) \left\{ \cfrac{1}{n+m} -\cfrac{1}{n(n-m-1)} \right\}\\
&+& 2 \left( \cfrac{n-m-1}{n} - \cfrac{1}{n+m}  \right) + 2 \left\{(m-1) \left(\cfrac{1}{n} - \cfrac{1}{n+m} \right) - \left( \cfrac{n-m-1}{n+m} - \cfrac{1}{n} \right) \right\} \\
&=& \cfrac{2(n^2-n-m)}{n(n+m)}.
\end{eqnarray*}
In the case where $m=n-2$, we change the definition of $A_2$ to the following:
\begin{eqnarray*}
A_2 (x, y)=
\begin{cases}
\cfrac{1}{2n-2}, & \mathrm{if}\ x=v_{1},\ y =v_{n-1},\\
\cfrac{1}{2n-2}, & \mathrm{if}\ x=y \in \Gamma(u_0) \cap \Gamma(v_{1}),\\
\cfrac{1}{2n-2}, & \mathrm{if}\ x =u_{n-1},\ y = u_0,\\
\cfrac{1}{n-2} \left( \cfrac{1}{n} - \cfrac{1}{2n-2} \right), & \mathrm{if}\ x \in \Gamma_m (v_0),\ y \in \left\{u_0, v_0, v_{n-1} \right\} \cup \Gamma_m (u_0) \setminus \left\{ v_1 \right\},\\
0, & \mathrm{otherwise}.
\end{cases}
\end{eqnarray*}
Then the 1-Wasserstein distance is estimated as follows:
\begin{eqnarray*}
W(m_{u_0}, m_{v_1}) &\leq& \cfrac{1}{2n-2} + \cfrac{1}{2n-2} + \left\{ 2+ 2(n-2) \right\} \left( \cfrac{1}{n}- \cfrac{1}{2n-2} \right) =\cfrac{2(n^2 -2n +2)}{n(n-2)}.
\end{eqnarray*}
Thus, this coincides with the 1-Wasserstein distance in the case of $m \geq n-3$.\\

On the other hand, we define a map $f_2 : V \to \mathbb{R}$ by
\begin{eqnarray*}
f_2 (w)=
\begin{cases}
2, & \mathrm{if}\ w \in V(K_n) \setminus \left\{u_{0} \right\},\\
1, & \mathrm{if}\ w \in \left\{u_0, v_0, v_1 \right\},\\
0, & \mathrm{otherwise}.
\end{cases}
\end{eqnarray*}
Given that $f$ is a 1-Lipschitz function, then by the Kantorovich duality, we have the following:
\begin{eqnarray*}
W(m_{u_0}, m_{v_1}) \geq  \cfrac{2(n-1)}{n+m} - \cfrac{2}{n} + \cfrac{2}{n+m} = \cfrac{2(n^2-n-m)}{n(n+m)}.
\end{eqnarray*}
Thus we obtain
\begin{eqnarray*}
W(m_{u_0}, m_{v_1})  = \cfrac{2(n^2-n-m)}{n(n+m)}= 1-\cfrac{(2+n)m + 2n - n^2}{n(n+m)}.
\end{eqnarray*}

%ケース２%
If $m < (-n + \sqrt{5n^2 -8n})/2$, then we define a map $A'_2: V \times V \to \mathbb{R}$ as follows:
\begin{eqnarray*}
A'_2 (x, y)=
\begin{cases}
\cfrac{1}{n+m}, & \mathrm{if}\ x=y \in \Gamma(u_0) \cap \Gamma(u_{n-1}),\\
\cfrac{1}{(n-m-1)(n+m)}, & \mathrm{if}\ x=v_{1}, y \in V(K'_n) \setminus \Gamma_m (u_0),\\
\cfrac{1}{n(n-m-1)}, & \mathrm{if}\ x \in V(K_n) \setminus \Gamma_m (v_0), y = u_0,\\
\cfrac{1}{n-m-1} \left( \cfrac{1}{n} - \cfrac{1}{n+m} \right), & \mathrm{if}\ x \in V(K_n) \setminus \Gamma_m (v_0),\\ &\ \ \ y \in \Gamma_m (u_0) \setminus \left\{ v_1 \right\},\\
\cfrac{1}{n-m-1} \left\{ \cfrac{1}{n+m} - \cfrac{1}{n-m-1} \left( \cfrac{m}{n} - \cfrac{m-1}{n+m} \right) \right\}, & \mathrm{if}\ x \in V(K_n) \setminus \Gamma_m (v_0),\\ &\ \ \  y \in V(K'_n) \setminus \Gamma_m (u_0),\\
\cfrac{1}{m} \left( \cfrac{1}{n} - \cfrac{1}{n+m} \right), & \mathrm{if}\ x \in \Gamma_m (v_0), y= v_0,\\
\cfrac{1}{n-m-1} \left\{ \cfrac{1}{n+m} - \cfrac{1}{m} \left( \cfrac{1}{n} - \cfrac{1}{n+m} \right) \right\}, & \mathrm{if}\ x \in \Gamma_m (v_0),\\ &\ \ \ y \in V(K'_n) \setminus \Gamma_m (u_0),\\
0, & \mathrm{otherwise}.
\end{cases}
\end{eqnarray*}
If $m < (-n + \sqrt{5n^2 -8n})/2$, then we have $(n-m-1)/n > m/(n+m) - (1/n - 1/(n+m))$. Therefore, $A'_2$ is a coupling between $m_{u_0}$ and $m_{v_1}$. Then the 1-Wasserstein distance is estimated as follows: 
\begin{eqnarray*}
W(m_{u_0}, m_{v_1}) &\leq& \cfrac{1}{n+m} +\cfrac{1}{n}+2(m-1) \left( \cfrac{1}{n} - \cfrac{1}{n+m} \right) \\
&+& 3(n-m-1) \left\{ \cfrac{1}{n+m} - \cfrac{1}{n-m-1} \left( \cfrac{m}{n} - \cfrac{m-1}{n+m} \right) \right\}\\
&+& \left( \cfrac{1}{n} - \cfrac{1}{n+m}\right) +2m \left\{ \cfrac{1}{n+m} - \cfrac{1}{m} \left( \cfrac{1}{n} - \cfrac{1}{n+m} \right) \right\} \\
&=& - \cfrac{m+2}{n} +\cfrac{3n-2}{n+m} = \cfrac{-m^2 - (2+n)m +3n^2 -4n}{n(n+m)}.
\end{eqnarray*}
On the other hand, we define a map $f'_2 : V \to \mathbb{R}$ by
\begin{eqnarray*}
f'_2 (w)=
\begin{cases}
3, & \mathrm{if}\ w \in V(K_n) \setminus (\left\{u_{0} \right\} \cup \Gamma_m (v_0)),\\
2, & \mathrm{if}\ w \in \left\{u_{0} \right\} \cup \Gamma_m (v_0),\\
1, & \mathrm{if}\ w \in \left\{v_{0} \right\} \cup \Gamma_m (u_0),\\
0, & \mathrm{otherwise}.
\end{cases}
\end{eqnarray*}
Given that $f'_2$ is a 1-Lipschitz function, then by the Kantorovich duality, we have
\begin{eqnarray*}
W(m_{u_0}, m_{v_1}) \geq  \cfrac{3(n-m-1)}{n+m} - \cfrac{2}{n} + \cfrac{2m}{n+m} + \cfrac{m+1}{n+m} - \cfrac{m}{n}= - \cfrac{m+2}{n} +\cfrac{3n-2}{n+m}.
\end{eqnarray*}
Thus we obtain
\begin{eqnarray*}
W(m_{u_0}, m_{v_1})  =  \cfrac{-m^2 - (2+n)m +3n^2 -4n}{n(n+m)} = 1- \cfrac{m^2 + 2(n+1)m - 2n^2 +4n}{n(n+m)}.
\end{eqnarray*}
This result means that $\kappa (u_0, v_1) \geq 0$ if and only if $m \geq -(n+1) + \sqrt{3n^2 - 2n +1}$. However, as we have $-(n+1) + \sqrt{3n^2 - 2n +1} \geq (-n + \sqrt{5n^2 -8n})/2$, there exists no $m$ with $\kappa (u_0, v_1) \geq 0$.
\end{proof}
\begin{remark}
\label{example2}
If we use Theorem \ref{Jost-Liuの評価}, then we obtain the following:
\begin{eqnarray*}
\kappa(u_0, v) \geq \cfrac{(m+2n)(m+2-n)}{n(n+m)},
\end{eqnarray*}
for any vertex $v \in \Gamma_m (u_0)$. 
\end{remark}

%４番目のリッチ曲率%
For the edge (4) in Figure \ref{case}, by using Theorem \ref{Jost-Liuの評価}, we obtain the following proposition:
\begin{proposition}
\label{４番目}
For any vertex $u \in V(K_n) \setminus \Gamma_m (v_0)$, we have
\begin{eqnarray*}
\cfrac{(2-n)m + n^2 -2n +2}{(n-1)(n+m)}  \geq \kappa(u_0, u) \geq 
\begin{cases}
\cfrac{(2-n)m +n^2 -3n +3}{(n-1)(n+m)}, & \mathrm{if}\ m < n-2,\\
\cfrac{1}{n-1}, & \mathrm{if}\ m=n-2.
\end{cases}
\end{eqnarray*}
In particular, $\kappa(u_0, u) \geq 0$ if and only if $m \leq n-2$.
\end{proposition}
\begin{proof}
We take $u_{n-1} \in V(K_n) \setminus \Gamma_m (v_0)$, and consider the Ricci curvature of $(u_0, u_{n-1})$. Given that the following equalities hold: 
\begin{eqnarray*}
1-\cfrac{1}{n+m}-\cfrac{1}{n-1}-\cfrac{n-2}{n-1} &=& - \cfrac{1}{n+m} <0, \\
1-\cfrac{1}{n+m}-\cfrac{1}{n-1}-\cfrac{n-2}{n+m} &=& \cfrac{(n-2)m -1}{(n+m)(n-1)} >0,
\end{eqnarray*}
we use Theorem \ref{Jost-Liuの評価} and obtain 
\begin{eqnarray*}
\kappa(u_0, u_{n-1}) &\geq& - \cfrac{(n-2)m -1}{(n+m)(n-1)} + \cfrac{n-2}{n+m}\\
&=& \cfrac{(2-n)m +n^2 -3n +3}{(n-1)(n+m)}.
\end{eqnarray*}
On the other hand, we define a map $f_4 : V \to \mathbb{R}$ as follows:
\begin{eqnarray*}
f_4 (w)=
\begin{cases}
2, & \mathrm{if}\ w \in \Gamma_m (u_0),\\
1, & \mathrm{if}\ w \in \left\{u_{0}, v_0, u_{n-1} \right\},\\
0, & \mathrm{otherwise}.
\end{cases}
\end{eqnarray*}
Given that $f_4$ is a 1-Lipschitz function, by Proposition \ref{kantoro}, we have the following:
\begin{eqnarray*}
W(m_{u_0}, m_{u_{n-1}}) \geq \cfrac{2m}{n+m} -\cfrac{1}{n-1} + \cfrac{2}{n+m}= \cfrac{2nm-3m + n-2}{(n-1)(n+m)}.
\end{eqnarray*}
Thus, we obtain
\begin{eqnarray*}
\kappa(u_0,u_{n-1})  \leq  \cfrac{(2-n)m + n^2 -2n +2}{(n-1)(n+m)} .
\end{eqnarray*}
This completes the proof.
\end{proof}

%６番目のリッチ曲率%
For the edges (5), (6) and (7) in Figure \ref{case}, by using Theorem \ref{Jost-Liuの評価} and Theorem \ref{Jost-Liuの評価２}, we obtain the following corollary as a consequence:
\begin{corollary}
\label{５から７番目}
\begin{enumerate}
\item For any vertex $u_i \in \Gamma_m (v_0)$, we have
\begin{eqnarray*}
\kappa(u_1, u_i) = \cfrac{n- 1}{n}.
\end{eqnarray*}
\item For any vertex $u'_i \in V(K_n) \setminus \Gamma_m (v_0)$ , we have
\begin{eqnarray*}
\kappa(u_1, u'_i) = \cfrac{n-2}{n}.
\end{eqnarray*}
\item For any two distinct vertices $u'_i$ and $u'_j \in V(K_n) \setminus \Gamma_m (v_0)$, we have
\begin{eqnarray*}
\kappa(u'_i, u'_j) = \cfrac{n-2}{n-1}.
\end{eqnarray*}
\end{enumerate}
\end{corollary}

%%応用%%
\section{Application}
In this section, we combine our result and the previous researches, and obtain some estimates of the eigenvalues of the normalized graph Laplacian and the Cheeger constant by the Ricci curvature on the $m$-gluing graph. Ollivier and Lin-Lu-Yau gave an estimate of the eigenvalues of the normalized graph Laplacian by the Ricci curvature as follows.
%正のRicci 曲率の性質%
\begin{theorem}[Ollivier \cite{Ol1}, Lin-Lu-Yau \cite{Yau1}]
\label{Ol}
Let $\lambda_{1}$ be the first non-zero eigenvalue of the normalized graph Laplacian $\Delta_{0}$. Suppose that $\kappa(x, y) \geq k$ for any edge $(x, y) \in E$ and for a positive real number $k$. Then we have
	\begin{eqnarray*}
	k \leq \lambda_{1} \leq 2 - k,
	\end{eqnarray*}
	where the normalized graph Laplacian is defined by $\Delta_{0} f (u)= \sum_{(v, u) \in E} ( f(u) - f(v) )$ for any vertex $u \in V$.
\end{theorem}
On the $m$-gluing graph, by Proposition \ref{３番目}, \ref{１番目}, \ref{２番目}, \ref{４番目}, and Corollary \ref{５から７番目}, for any edge $(x,y) \in E(K_n +_M K'_n)$ we have 
\begin{eqnarray*}
\kappa(x, y) \geq 
\begin{cases}
\cfrac{n-6}{n(2n-3)}, & \mathrm{if}\ n >6,\\
\cfrac{n-2}{n(n-1)}, & \mathrm{if}\ n \in \left\{5, 6\right\},
\end{cases}
\end{eqnarray*}
where $M := \min \left\{ m \mid \kappa (x, y)>0\ \mathrm{for\ any\ edge}\ (x,y) \in E(K_n +_m K'_n) \right\}$, that is, $M=n-3$ if $n>6$, and $M=n-2$ if $n \in \left\{5, 6\right\}$. By Theorem \ref{Ol}, we obtain an estimate of the first non-zero eigenvalue of the normalized graph Laplacian. In addition, the first non-zero eigenvalue of the normalized graph Laplace operator is related to the Cheeger constant as follows.
\begin{definition}
The {\em Cheeger constant} of a graph $G$, denoted by $h(G)$, is defined by
\begin{eqnarray*}
h(G)=  \min \left\{ \cfrac{|\partial A|}{|A|}\ \middle|\ A \subset V(G), 0 < |A| < |V(G)|/2 \right\},
\end{eqnarray*}
where $\partial A = \left\{ (x, y) \in E \mid x \in A\ \mathrm{and}\ y \in V(G)\setminus A \right\}$.
\end{definition}
\begin{theorem}[Chung \cite{ch}]
\label{Ch}
Let $\lambda_{1}$ be the first non-zero eigenvalue of the normalized graph Laplace operator $\Delta_{0}$. Then we have the following:
\begin{eqnarray*}
\cfrac{\lambda_{1}}{2} \leq h(G) \leq \sqrt{2 \lambda_{1}}.
\end{eqnarray*}
\end{theorem} 
Thus, if we combine Theorem \ref{主結果} and Theorem \ref{Ch}, then we obtain the following corollary:
\begin{corollary}
For the $M$-gluing graph $K_n +_M K'_n$, we have
\begin{eqnarray*}
\begin{cases}
\cfrac{n-6}{2n(2n-3)} \leq h(K_n +_M K'_n) \leq \sqrt{\cfrac{2(4n^2 -7n +6)}{n(2n-3)}}, & \mathrm{if}\ n>6,\\
\cfrac{n-2}{2n(n-1)} \leq h(K_n +_M K'_n) \leq\sqrt{\cfrac{2(n^2 -2n +2)}{n(n-1)}}, & \mathrm{if}\ n \in \left\{5, 6\right\}.
\end{cases}
\end{eqnarray*}
\end{corollary}

\end{document}